\documentclass[11pt, oneside]{article} 
 
\usepackage{hyperref}
\usepackage{geometry}                		
\usepackage{graphicx}														
\usepackage{amssymb}
\usepackage{amsmath, amssymb, graphics, setspace}
\newcommand{\mathsym}[1]{{}}
\newcommand{\unicode}[1]{{}}

\usepackage{amsmath}
\usepackage{wrapfig}    
\usepackage{graphicx}
\usepackage{amsthm}
\usepackage{bbm}
\usepackage[dvipsnames]{xcolor}
\usepackage{bm} 
\usepackage{appendix}

\newtheorem{corollary}{Corollary}

\usepackage[utf8]{inputenc}
\usepackage[dvipsnames]{xcolor}
\usepackage{fancyhdr}
\usepackage{mathrsfs}
\newtheorem{theorem}{Theorem}
\newtheorem{definition}{Definition}[section]

\title{Proofs For Progressively Generalized Fibonacci Identities Using Maximal Independent Sets of Tree Graphs}
\author{Connor Oxenhorn}
\date{April 26, 2022}

\begin{document}
\maketitle
\begin{abstract}
This paper generalizes a graph theoretic proof technique for a Fibonacci identity proposed by Lee Knisley Sanders, and explores characteristics of these generalized theorems ad infinitum.
\end{abstract}
\section{Introduction}
Finding unexpected connections across various areas of mathematics is exciting and can offer new perspectives on various thoroughly-studied concepts. In this spirit, we will look at how one of the most-studied sequences—the Fibonacci sequence—makes an elegant appearance within graph theory. We initially present Sanders' graph theoretic proof of a Fibonacci identity and then augment it to prove progressively-generalized versions of the identity. Finally, we entertain an exploration of the consequences of continuing this generalizing pattern of the identities ad infinitum.
\subsection{Fibonacci Sequences}

\begin{definition}[Fibonacci Sequence]
Suppose $\alpha,\beta\in\mathbb{R}$. Then we say a sequence $\{G_n\}_{n\in\mathbb{N}}$ is a Fibonacci sequence if $$G_n = 
\begin{cases}
	\alpha & \text{$n = 0$}\\
    	\beta & \text{$n = 1$}\\
      	G_{n-1} + G_{n-2} & \text{$n > 1$}\\
\end{cases}.$$ 
\end{definition}

We will refer to $G_0,G_1$ as the \textit{seed terms}, and $\alpha,\beta$ as the \textit{seed values} of a Fibonacci sequence (respectively). 
\par
Additionally, two \textit{types} of Fibonacci sequences will be referenced frequently. The first is $\{F_n\}_{n\in\mathbb{N}}$, sometimes referred to as the \textit{classical Fibonacci sequence}, with seed values $\alpha = 0,\beta = 1$. The second is a generalized Fibonacci sequence $\{G_n\}_{n\in\mathbb{N}}$ with seed values $\alpha,\beta \in \mathbb{R}$. Note that these are not mutually-exclusive classifications, i.e. $$\{F_n\}_{n\in\mathbb{N}}\in\big\{\{G_n\}_{n\in\mathbb{N}} : \alpha,\beta\in\mathbb{R}\big\}.$$

\subsection{Graph Theoretic Definitions}

\begin{definition}[Undirected Simple Graph]
Let $G=(V,E)$ be an ordered pair consisting of a vertex set $V$, and an edge set $E = \big\{(x,y)\in V\times V : x\neq y,(x,y)=(y,x) \big\}.$ Then we call $G$ an undirected simple graph.
\end{definition}

\begin{definition}[Vertex Degree]
Given an undirected simple graph $G=(V,E)$, we define the degree of a vertex $x\in V$ as $$\rho(x) = \big\vert\{(i,j)\in E : x=i \text{ or } x=j\}\big\vert.$$
\end{definition}

\begin{definition}[Path]
Let $G=(V,E)$ be an undirected simple graph. We call an ordered set of edges $\{(x_0,x_1),(x_1,x_2),\dots,(x_{m-2},x_{m-1}),(x_{m-1},x_m)\}$ a path on $G$.
\end{definition}

\begin{definition}[Tree]
Let $G=(V,E)$ be an undirected simple graph. $G$ is connected if there exists a path between each $x,y\in V$. $G$ is acyclic if there is no non-self-intersecting path on $G$ that begins and ends at the same vertex. We call a connected, acyclic undirected simple graph a tree.
\end{definition}

Fittingly, a vertex $x$ of a tree $G$ is called a ``leaf" if it has the property that $\rho(x)=1.$

\begin{definition}[Maximal Independent Set]
Let $G=(V,E)$ be an undirected simple graph. An independent set of vertices on $G$ is a subset $\mathcal{V}\subseteq V$ such that $(x,y)=(y,x)\not\in E$ for all $x,y\in\mathcal{V}.$ A maximal independent set (MIS) on $G$ is an independent set that is not contained in any other independent set on $G$.
\end{definition}
\section{Sanders' Tree}
We begin with the $n$-vertex tree $T$ (shown below in blue), referred to as the \textit{core tree}. Now, let $p(T)$, called the expanded tree, be the tree created by adding a leaf to each vertex of $T$ (shown below in red and blue).
\begin{center}
\includegraphics[width=0.8\textwidth]{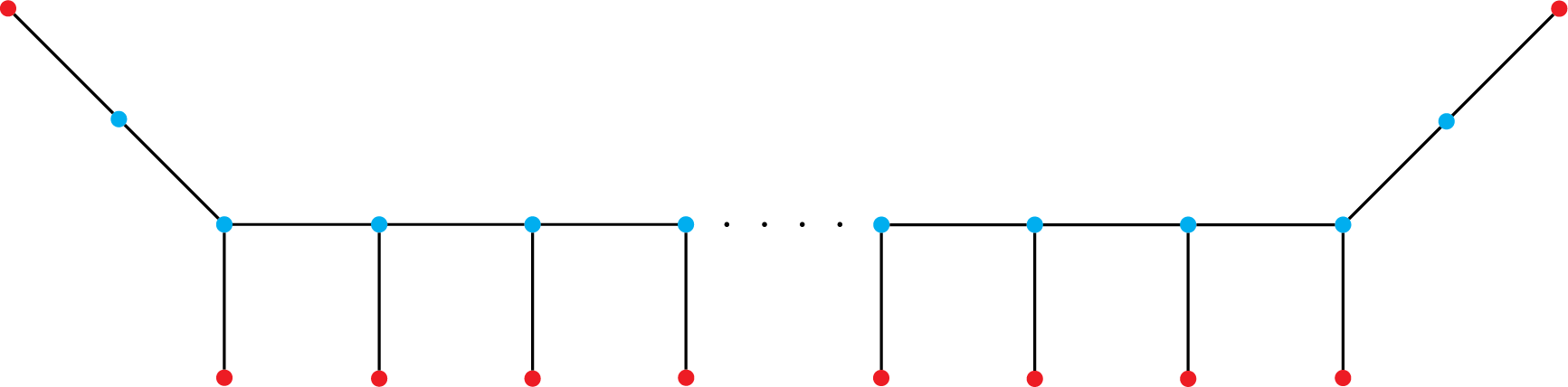}
\end{center}
We call the central $n-2$ vertices of the core tree the \textit{central path} (shown below in green). Each vertex of the central path is labelled $v_i$, and each leaf of the expanded tree adjacent to $v_i$ is labelled $z_i$ as illustrated below for $i=1,2,\dots,n-2$. 
\newline
\begin{center}
\includegraphics[width=0.8\textwidth]{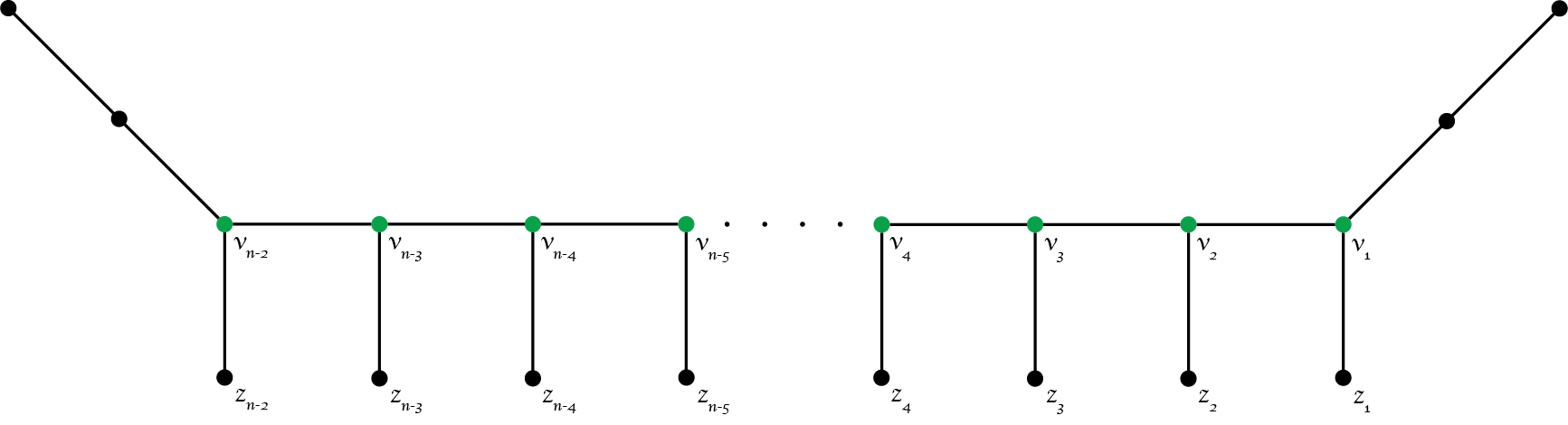}
\end{center}
Next, let $\lambda(x)$ be the number of maximal independent sets (MISs) of $p(T)$ that contain the vertex $x$, as well as $r(x),l(x)$ be the number of MISs that contain $x$ and only the vertices to the right, left of $x$ (respectively). Note that vertices ``to the right" or ``to the left" of a leaf adjacent to a vertex in the central path include its adjacent vertex in the central path. 
\par
Finally, we label each vertex of the central path and its adjacent leaves from left to right with the value of $l(x)$. This is the graph construction and labeling scheme Sanders presents, which we will generalize in following sections.
\newline
\begin{center}
\includegraphics[width=0.6\textwidth]{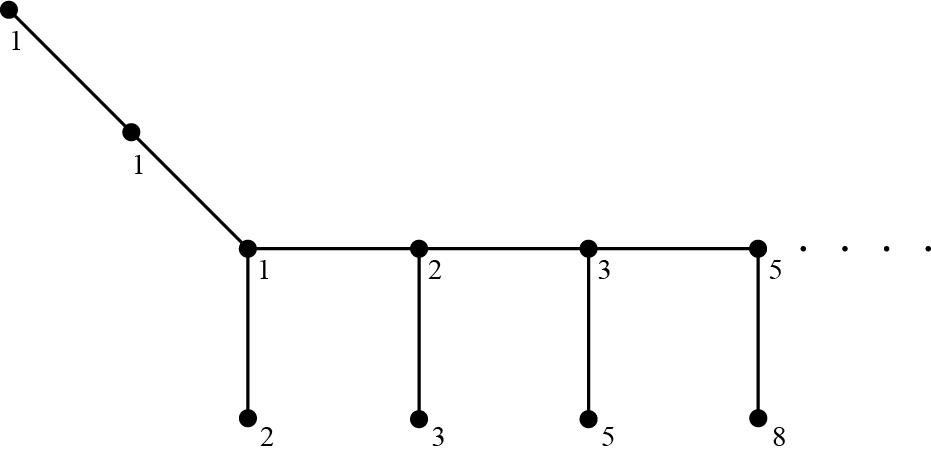}
\end{center}
\section{Generalizing Sanders' Proof}
\subsection{Fibonacci Identities}
\begin{theorem}[Ruggles, 1963]\label{thm_gf}
Let $G_n\in\{G_n\}_{n\in\mathbb{N}}$ and $F_i\in\{F_i\}_{i\in\mathbb{N}}$. Then for all $1 \leq i \leq n,$ $$G_n = G_{n-i+1}F_i + G_{n-i}F_{i-1}.$$
\end{theorem}

We also note a special case of Theorem~\ref{thm_gf} where $i=n$
\begin{equation}
\label{eqn:special_case}
G_n = G_{1}F_n + G_{0}F_{n-1} = \beta F_n + \alpha F_{n-1}.
\end{equation}

\begin{corollary}\label{thm_ff}
Let $F_j\in\{F_j\}_{j\in\mathbb{N}}$. Then for all $1 \leq i \leq n,$ $$F_n = F_{n-i+1}F_i + F_{n-i}F_{i-1}.$$
\end{corollary}

\begin{theorem}\label{thm_gg}
Let $G_j\in\{G_j\}_{j\in\mathbb{N}}$. Then for all $1 \leq i \leq n,$ $$\beta G_{n} + \alpha G_{n-1} = G_{n-i+1}G_i + G_{n-i}G_{i-1}.$$
\end{theorem}
\subsection{Proofs}
In this section we outline Sanders' graph theoretic proof of Corollary~\ref{thm_ff}, before generalizing the technique to prove Theorem~\ref{thm_gf} and Theorem~\ref{thm_gg}.
\paragraph{Corollary 1}
\begin{proof}
Let $M_{p(T)}$ be the total number of MISs of an expanded tree. Sanders proves Corollary~\ref{thm_ff} using three results. The first is $M_{p(T)} = \lambda(z_i) + \lambda(v_i)$ for any leaf $z_i$ and its adjacent vertex in the central path $z_i$ \cite{sanders1990aproof}. The second is $\lambda(v_i) = l(v_i)r(v_i), \lambda(z_i) = l(z_i)r(z_i)$; this boils down to a simple combinatorial argument. The third is $r(v_i) = F_{i+1}$, $l(v_i) = F_{n-i}$ and $r(z_i) = F_{i+2}$, $l(z_i) = F_{n-i+1}$ \cite{sanders1990aproof}.
\par
Using the above results with the vertices $v_{n-2}$ and $z_{n-2}$, it can be shown \cite{sanders1990aproof} $$M_{p(T)} = 2F_n + F_{n-1} = F_{n+2}.$$ Using this in conjunction with the second and third results, we get the following formula formula \cite{sanders1990aproof} $$F_{n+2} = F_{n-i+1}F_{i+2} + F_{n-i}F_{i+1}$$ $$\text{for all }1 \leq i \leq n.$$
We can adjust the subscripts $i \rightarrow (i-2)$ and $(n+2) \rightarrow n$ to get Corollary~\ref{thm_ff}.
\end{proof}
\paragraph{Theorem 1}\label{thm_gf_pf}
\begin{proof}
We will follow a similar process to the above, however we tweak it by introducing \textit{weighted left MIS counts}, denoted $\mathscr{L}(x_i)$ for a vertex $x_i$, and defined by $\mathscr{L}(x_i) = \beta l(x_i) + \alpha l(x_{i-1})$ for any $\alpha,\beta\in\mathbb{R}$. Using (\ref{eqn:special_case}) and Sanders' third result, we know that $\mathscr{L}(v_i) = \beta F_{n-i} + \alpha F_{n-i-1} = G_{n-i}$ and $\mathscr{L}(z_i) = \beta F_{n-i+1} + \alpha F_{n-i} = G_{n-i+1}$ for some generalized Fibonacci sequence $\{G_j\}_{j\in\mathbb{N}}$ with $G_0 = \alpha, G_1 = \beta$ as its seed terms and seed values. It is worth noting that the method Sanders uses to prove Corollary \ref{thm_ff} inherently assumes the weighted left MIS counts to have weights $\alpha=0,\beta=1$ which are the seed values of the classical Fibonacci sequence.
\par
We now re-label the central path of $p(T)$ and its adjacent leaves from left to right in the same manner as before, but changing each vertex label from $l(x)$ to $\mathscr{L}(x)$. Note that we are only weighting the left MIS counts, not the right MIS counts.
\begin{center}
\includegraphics[width=0.6\textwidth]{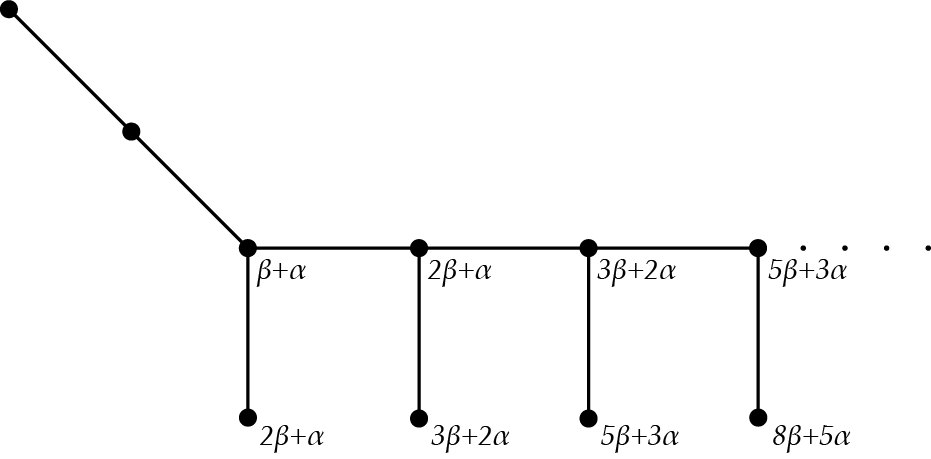}
\end{center}
Using Sanders' three results and (\ref{eqn:special_case}) we find that $$M_{p(T)} = 2(\beta F_{n} + \alpha F_{n-1}) + (\beta F_{n-1} + \alpha F_{n-2}) = 2G_n + G_{n-1} = G_{n+2}.$$ Additionally, $$M_{p(T)} = \lambda(z_i) + \lambda(v_i) = \mathscr{L}(z_i)r(z_i) + \mathscr{L}(v_i)r(v_i) = G_{n-i+1}F_{i+2} + G_{n-i}F_{i+1}.$$ Combining these results, we get $$G_{n+2} = G_{n-i+1}F_{i+2} + G_{n-i}F_{i+1}$$ $$1 \leq i \leq n.$$
Again, we can adjust the value of $i \rightarrow (i-2)$ and $(n+2) \rightarrow n$ to get Theorem 1.
\end{proof}
\par
We should also note that we get the same result by weighting the $r(x_i)$ instead of the $l(x_i)$. This is because the products on the right-hand side of Theorem~\ref{thm_gf} go from $G_nF_{1} \rightarrow G_1F_n$ and $G_{n-1}F_0 \rightarrow G_0F_{n-1}$ as $i$ goes from $1\to n$. In this way, choosing $r(x_i)$ or $l(x_i)$ to be the weighted MIS counts is simply a matter of choosing which end of the range of $i$ we start with (i.e. $1 \rightarrow n$ or $n \rightarrow 1$).
\paragraph{Theorem 2}
\begin{proof}
By weighting both the $l(x)$ and $r(x)$ and following the method used in the proof of Theorem~\ref{thm_gf}, we get $$M_{p(T)} = G_{n-i+1}G_{i+2} + G_{n-i}G_{i+1}.$$ Using the second result with the vertices $v_{n-2}$ and $z_{n-2}$, we also find that $$M_{p(T)} = G_3G_n+G_2G_{n-1}.$$ Combining these two equations, we get $$G_3G_n+G_2G_{n-1} = G_{n-i+1}G_{i+2} + G_{n-i}G_{i+1}$$ $$1 \leq i \leq n.$$
We can again adjust the range of $i$ and $n$ to get
$$G_3G_{n-2} + G_2G_{n-3} = G_{n-i+1}G_i + G_{n-i}G_{i-1}$$ $$1 \leq i \leq n.$$
Lastly, using the result of the above equation where $i=n$, we know $$G_3G_{n-2} + G_2G_{n-3} = G_1G_{n} + G_0G_{n-1} = \beta G_{n} + \alpha G_{n-1}.$$
Hence, we get $$\beta G_{n} + \alpha G_{n-1} = G_{n-i+1}G_i + G_{n-i}G_{i-1}.$$ 
\end{proof}
\subsection{Discussion}\label{discussion}
These theorems represent an elegant example of progression in generalization. Working backwards, 
$$\beta G_{n} + \alpha G_{n-1} = G_{n-i+1}G_i + G_{n-i}G_{i-1}$$ is the more general case of $$G_n = G_{n-i+1}F_i + G_{n-i}F_{i-1}$$ where, in the latter case, the values of $\alpha,\beta$ are $0,1$ respectively. Furthermore, $$F_n = F_{n-i+1}F_i + F_{n-i}F_{i-1}$$ is a special case of the above special case where the Fibonacci terms all belong to $\{F_n\}_{n\in\mathbb{N}}$. As is the case with Theorem~\ref{thm_gf}, writing Corollary~\ref{thm_ff} as $$\beta F_{n} + \alpha F_{n-1} = F_{n-i+1}F_i + F_{n-i}F_{i-1}$$ would be redundant as $\alpha=0,\beta=1$.
\par
More abstractly, it is as if these related theorems contain an intersection of two different Fibonacci sequences. As these two Fibonacci sequences become generalized (i.e. their base values are no longer $0,1$), an adjustment needs to be made wherein the left-hand side of the identity becomes progressively more generalized as well, incorporating more of the base values of the generalized sequence(s). This first  happens with the `generalization' of the individual term(s) on the left-hand side, and then the `generalization' of the constants $\alpha,\beta$ that are attached to these term(s).This can occur for none, one, or both of the Fibonacci sequences incorporated into these theorems, and this can be seen in all three theorems above.
\par
It is also imperative to note that $\alpha,\beta$ in Theorem~\ref{thm_gg} may be different from the base values of the sequence to which the terms on the left-hand side of the identity belong. Because we can interpret this theorem as containing two different Fibonacci sequences—as is apparent in the above equations as well—it may be more appropriate to re-write Theorem~\ref{thm_gg} as  
\begin{equation}\label{eqn: gen_case}
\beta^\prime G_{n} + \alpha^\prime G_{n-1} = G_{n-i+1}G_i^\prime + G_{n-i}G_{i-1}^\prime
\end{equation}
$$\text{for all } 1 \leq i \leq n.$$
\section{Infinite Generalization}
The above begs the question: how far can we generalize this technique? What if, in the same way in which we weighted the left (and right) MIS counts, we weight the weighted MIS counts? For notational simplicity, we use the statement of Theorem~\ref{thm_gg} as originally given, not the interpretation given in (\ref{eqn: gen_case}). However, one can easily verify that the below applies to this broader interpretation of the theorem as well. 
\par
Letting $X^{(k)}_n$ denote the value of the left-hand side of the $k^{th}$-generalized version of Theorem~\ref{thm_gg}, such a progression would look like $$X^{(0)}_n = \beta G_{n} + \alpha G_{n-1} = G_{n-i+1}G_i + G_{n-i}G_{i-1}$$ $$X^{(1)}_n = X^{(0)}_{n-i+1}G_i + X^{(0)}_{n-i}G_{i-1}$$ $$X^{(2)}_n = X^{(0)}_{n-i+1}X^{(0)}_{i} + X^{(0)}_{n-i}X^{(0)}_{i-1}$$ $$X^{(3)}_n = X^{(1)}_{n-i+1}X^{(0)}_{i} + X^{(1)}_{n-i}X^{(0)}_{i-1}$$ $$\vdots$$ $$X^{(2k)}_n = X^{(2k-k-1)}_{n-i+1}X^{(2k-k-1)}_{i} + X^{(2k-k-1)}_{n-i}X^{(2k-k-1)}_{i-1}$$ $$X^{(2k+1)}_n = X^{(2k-k)}_{n-i+1}X^{(2k-k-1)}_{i} + X^{(2k-k)}_{n-i}X^{(2k-k-1)}_{i-1}$$ $$\vdots$$
\par
The question is, what are these $X^{(k)}_n$ terms and do they have any generalizable characteristics? 
\subsection{Fibonacci Terms in Disguise}
To answer this, we begin by noting that $$X^{(0)}_{n-1} + X^{(0)}_{n-2} = (\beta G_{n-1} + \alpha G_{n-2}) + (\beta G_{n-2} + \alpha G_{n-3})$$ $$ = \beta(G_{n-1} + G_{n-2}) + \alpha(G_{n-2} + G_{n-3}) = \beta G_{n} + \alpha G_{n-1} = X^{(0)}_{n}.$$ Thus, inductively using Theorem~\ref{thm_gg}, we see that for some fixed $k\in\mathbb{N}$, $\{X^{(k)}_n\}_{n\in\mathbb{N}}$ is a Fibonacci sequence. Furthermore, consider a sequence $\{Y_n\}_{n\in\mathbb{N}}$ defined by the linear combination $$Y_n = a_1X^{(k_1)}_{n_1} + a_2X^{(k_2)}_{n_2} + \dots + a_mX^{(k_m)}_{n_m}.$$
\begin{theorem}\label{thm_yn}
$\{Y_n\}_{n\in\mathbb{N}}$ is a Fibonacci sequence in $n$ for all $k_i,n_i\in\mathbb{N},a_i\in\mathbb{R}$.
\end{theorem}
\begin{proof}
$$Y_{n-1} + Y_{n-2} $$ $$= (a_1X^{(k_1)}_{n_1-1} + a_2X^{(k_2)}_{n_2-1} + \dots + a_mX^{(k_m)}_{n_m-1}) + (a_1X^{(k_1)}_{n_1-2} + a_2X^{(k_2)}_{n_2-2} + \dots + a_mX^{(k_m)}_{n_m-2}) $$ $$=a_1(X^{(k_1)}_{n_1-1} + X^{(k_1)}_{n_1-2}) + a_2(X^{(k_2)}_{n_2-1} + X^{(k_2)}_{n_2-2}) + \dots + a_m(X^{(k_m)}_{n_m-1} + X^{(k_m)}_{n_m-2}) $$ $$=a_1X^{(k_1)}_{n_1} + a_2X^{(k_2)}_{n_2} + \dots + a_mX^{(k_m)}_{n_m}$$ $$= Y_n.$$ 
\end{proof}
\subsection{A Meta Characteristic}
We have seen that the above iterated procedure is such that it preserves the generalized Fibonacci structure for the left-hand side terms, which is somewhat remarkable. A final question remains concerning whether these generalized left-hand terms are Fibonacci in their generalized-ness (loosely speaking). More concretely, we explore the question $$X^{(k)}_{n} \stackrel{?}{=} X^{(k-1)}_{n} + X^{(k-2)}_{n}.$$
We define the sequence $\{Y^{(k)}\}_{k\in\mathbb{N}}$ to be such that $$Y^{(k)} = a_1X^{(k_1)}_{n_1} + a_2X^{(k_2)}_{n_2} + \dots + a_mX^{(k_m)}_{n_m}.$$
\begin{theorem}\label{thm_xnk}
$\{Y^{(k)}\}_{k\in\mathbb{N}}$ is a Fibonacci sequence in $k$ for all $k_i,n_i\in\mathbb{N},a_i\in\mathbb{R}$ if and only if the seed values of the terms $G_j\in\{G_j\}_{j\in\mathbb{N}}$ in the left- and right-hand side products in $\{X^{(0)}_n\}_{n\in\mathbb{N}}$ are either $\alpha,\beta=0$, $\alpha,\beta=1$, or $\alpha=-1,\beta=0$.
\end{theorem}
\begin{proof}
We begin by proving the case where $Y^{(k)} = X^{(k)}_n$, as the more general case for $Y^{(k)} = a_1X^{(k_1)}_{n_1} + a_2X^{(k_2)}_{n_2} + \dots + a_mX^{(k_m)}_{n_m}$ follows trivially.
\par
In the case where $k=2$,
$$X^{(1)}_n + X^{(0)}_n = \beta X^{(0)}_n + \alpha X^{(0)}_{n-1} + \beta G_n + \alpha G_{n-1} $$
\begin{equation}\label{eqn:sys1}
=(\beta+\beta^2)G_n + (\alpha+2\alpha\beta)G_{n-1} + \alpha^2G_{n-2}.
\end{equation}
We compare this with $$X^{(2)}_n =  X^{(0)}_1X^{(0)}_n + X^{(0)}_0X^{(0)}_{n-1} $$ 
\begin{equation}\label{eqn:sys2}
=(\beta^3+\alpha^2\beta)G_n + (\alpha^3+3\alpha\beta^2-\alpha^2\beta)G_{n-1} + (\beta\alpha^2+\alpha^2\beta-\alpha^3)G_{n-2}.
\end{equation}
Note that $(\alpha,\beta)=(0,0)$ is a trivial solution to (\ref{eqn:sys1}) and (\ref{eqn:sys2}). This leaves us with the 3-part system of equations: $$1+\beta = \beta^2+\alpha^2,$$ $$1+2\beta=\alpha^2+3\beta^2-\alpha\beta,$$ $$\text{and}$$ $$2\beta-\alpha=1$$ which are satisfied if and only if $(\alpha,\beta) = (1,1)$ or $(\alpha,\beta) = (-1,0)$.
\par
Finally, it follows from a simple inductive argument that $X^{(k)} = X^{(k-1)} + X^{(k-2)}$ for all $k\in\mathbb{N}$ if and only if $\alpha,\beta=0$, $\alpha,\beta=1$, or $\alpha=-1,\beta=0$.
\end{proof}

\section{Acknowledgements}
I would like to thank Liyang Zhang for his feedback, as well as for forcing me out of my comfort zone to explore ways to engage with other peoples' contributions to the thoroughly-investigated topic of Fibonacci sequences.
\nocite{ruggles1963some}
\nocite{introduction}
\bibliography{citations}{}
\bibliographystyle{ieeetr}

\end{document}